\documentclass[a4paper]{amsart}
\usepackage{color}
\usepackage{subcaption}
\usepackage{booktabs}
\usepackage{longtable}
\usepackage{amscd}
\usepackage{amssymb}
\usepackage{amsmath}
\usepackage{xypic}
\usepackage{nicefrac}
\usepackage{color}
\usepackage{tikz}
\usepackage{hyperref}
\usepackage[draft,layout={margin,index}]{fixme}
\fxsetup{theme=color}

\newtheorem{theorem}{Theorem}[section] 

\newtheorem{lemma}[theorem]{Lemma}
\newtheorem{proposition}[theorem]{Proposition}

\newtheorem{corollary}[theorem]{Corollary}

{\theoremstyle{remark}
\newtheorem{remark}[theorem]{Remark}
\newtheorem{example}[theorem]{Example}

\newtheorem{question}{Question}
}

\theoremstyle{definition}

\newtheorem{definition}[theorem]{Definition}

\newcommand{\CC}{\mathbb{C}}

\newcommand{\RR}{\mathbb{R}}

\newcommand{\QQ}{\mathbb{Q}}
\newcommand{\ZZ}{\mathbb{Z}}
\newcommand{\NN}{\mathbb{N}}

\newcommand{\PP}{\mathbb{P}}

\newcommand{\CO}{{\mathcal{O}}}

\newcommand{\X}{{\mathcal{X}}}

\newcommand{\D}{{\mathfrak{D}}}

\renewcommand{\u}{{\mathbf{u}}}

\DeclareMathOperator{\Fut}{Fut}

\DeclareMathOperator{\tail}{tail}

\DeclareMathOperator{\pos}{pos}
\DeclareMathOperator{\vol}{vol}
\DeclareMathOperator{\Cl}{Cl}

\DeclareMathOperator{\face}{face}

\title{On irregular Sasaki-Einstein metrics in dimension~$5$}

\author[H. S\"u{\ss}]{Hendrik S\"u\ss}
\address{Hendrik S\"u\ss\\ School of Mathematics,
The University of Manchester,
Alan Turing Building,
Oxford Road
Manchester M13 9PL}
\email{\href{mailto:hendrik.suess@manchester.ac.uk}{hendrik.suess@manchester.ac.uk}}

\begin{document}
\maketitle
\begin{abstract}
  We show that there are no irregular Sasaki-Einstein structures on rational homology 5-spheres. On the other hand, using K-stability we prove the existence of continuous families of non-toric irregular Sasaki-Einstein structures on odd connected sums of $S^2 \times S^3$.
\end{abstract}


\section{Introduction}
In this paper we study Fano cone singularities $X$ polarised by a Reeb field $\xi$ in the sense of Collins-Sz\'ekelyhidi. The choice of such a polarisation induces a Sasakian metric structure on the link of $X$. This Sasakian structure is called \emph{quasi-regular} if the Reeb field generates a one-dimensional torus action it is called \emph{irregular} if the corresponding torus is at least two-dimensional.

In their paper \cite{collins2015sasaki}  Collins and Sz\'ekelyhidi show that for an appropriate scaling of $\xi$ the K-stability of the polarised cone singularity $(X,\xi)$ is equivalent to the existence of a Ricci-flat cone K\"ahler metric or equivalently the existence of a Sasaki-Einstein structure  with Reeb field $\xi$ on the link of $X$. For a long time only quasi-regular Sasaki-Einstein structures were known. Eventually in \cite{zbMATH02201007, zbMATH06085780} irregular toric Sasaki-Einstein structures on $S^2 \times S^3$ were found. Further examples were studied e.g. in \cite{1126-6708-2005-08-054,OOTA2007377,zbMATH05306068,zbMATH05682662}. In \cite{zbMATH05034151} the construction of \cite{zbMATH02201007} was generalised to produce higher dimensional non-toric examples. However, in particular in dimension $5$ there are still open problems connected to the existence of irregular Sasaki-Einstein structures. In \cite{zbMATH05556823} and \cite{zbMATH06085995}, respectively,  Sparks asks the following two questions.
\begin{question}
  Are there continuous families of irregular Sasaki-Enstein structures in dimension $5$?
\label{quest:irreg-families}
\end{question}
\begin{question}
  Are there any non-toric irregular Sasaki-Einstein structures in dimension $5$?
\label{quest:irreg-non-toric}
\end{question}
In \cite{collins2015sasaki} Collins and Sz\'ekelyhidi ask 
\begin{question}
  Are there irregular Sasaki-Einstein on $S^5$?
\label{quest:irreg-sphere}
\end{question}
The purpose of this article is to answer these three questions. Our main tools are the combinatorial description of torus actions via polyhedral divisors, which has been developed by Altmann and Hausen in \cite{pre05013675}, and the results of Collins-Sz\'ekelyhidi from \cite{collins2015sasaki}. The paper is organised as follows. In Section~\ref{sec:spheres} we derive the non-existence of irregular Sasaki-Einstein structures on the $5$-sphere from a more general result. Hence, we give a negative answer to Question~\ref{quest:irreg-sphere}. In Section~\ref{sec:k-stability-fano} we recall the notion of K-stability for polarised cone singularities. In Section~\ref{sec:polyhedral-divisors} we consider cone singularities arising from polyhedral divisors and study their K-stability. These results are then used in Section~\ref{sec:hom-spheres} to generalise the result on the non-existence of irregular Sasaki-Einstein structures on $S^5$ to the case of rational homology $5$-spheres. In Section~\ref{sec:main-example} we study a family of non-toric irregular Sasaki-Einstein metrics on the connected sums $(2\ell+1)(S^2 \times S^3)$, where for $\ell > 1$ we find non-trivial moduli for these structures. Hence, this example serves as an affirmative answer to both Question~\ref{quest:irreg-families} and \ref{quest:irreg-non-toric}.

\subsection*{Acknowledgments}
 I would like to thank G\'abor Sz\'ekelyhidi for pointing me to this problem, Tristan Collins, Chi Li and James Sparks for helpful conversation on the subject and Joaqu\'in Moraga for helping with the calculations of the fundamental group for the considered examples.

\section{Sasaki-Einstein structures of complexity $1$ on spheres}
\label{sec:spheres}
Let us first fix some notation and then recall the definition of \emph{polarised Fano cone singularities} from \cite{collins2015sasaki}. Fix an algebraic torus $T$. We are going to denote its character lattice by $M$ and its dual lattice by $N$ with their natural pairing $\langle\,,\,\rangle$. Moreover, by $M_\RR$ and $N_\RR$ we denote the corresponding vector spaces.

\begin{definition}
  By a \emph{cone singularity} we mean a normal complex affine variety $X\subset \CC^N$, such that an algebraic torus $T$ acts on $X$ with positive weights. This corresponds to an $M$-grading of the affine coordinate ring 
\[\CC[X]=R=\bigoplus_{u \in M} R_u,\]
such that $R_0=\CC$. An element $\xi \in N_\RR$ is called a \emph{Reeb field} or a \emph{polarisation} of $X$ if it fulfils $\langle u, \xi\rangle > 0$ for all $u \in M \setminus\{0\}$ and $R_u \neq 0$. The Reeb fields form a cone in $N_\RR$ which is called the \emph{Reeb cone} of $X$. Its closure $\sigma \subset N_\RR$ is a pointed, rational and polyhedral cone of full dimension.

The $X$ is called a \emph{Fano cone singularity} if $X$ is additionally $\QQ$-Gorenstein and log-terminal.

The difference $(\dim X - \dim T)$ is called the \emph{complexity} of the cone singularity. Cone singularities of complexity $0$ are called \emph{toric}.
\end{definition}

The condition of having isolated singularities for $X \subset \CC^N$ implies that the link $L=X \cap S^{2N-1}$ is a smooth manifold. Moreover, the choice of $\xi$ induces an Sasakian metric structure on $L$. The log-terminality corresponds to the positivity of this Sasakian metric structure. On the other hand every, positive Sasakian structure on a compact manifold arises as the link of an isolated Fano cone singularity. This is essentially a consequence of \cite[Thm.~8.2.18]{zbMATH05243165}. 

\begin{definition}
  The corresponding Sakaki metrics structures on the link of $X$ are called \emph{quasi-regular} if the Reeb field $\xi \in M_\RR$ is rational and \emph{irregular} otherwise.
\end{definition}

We are also interested in the topology of these links. Hence, we are going to derive basic information on their fundamental group and their homology groups.

\begin{lemma}
\label{lem:sphere-homology}
  Let $X$ be an isolated Fano cone singularity. Let us denote its link by $L$. Then $\Cl(X) \cong H^2(L,\ZZ)$. In particular, if $L$ is a sphere, then $X$ is factorial and if $L$ is a rational homology sphere, then $X$ is $\QQ$-factorial.
\end{lemma}
\begin{proof}
  By our preconditions a generic one-parameter subgroup of $T$ will induce a \emph{good} $\CC^*$-action on $X$ with the singular point $0$ as its unique fixed point. Hence, $L$ is a deformation retract of $X \setminus \{0\}$. On the other hand, it follows from Satz~6.1 in \cite{zbMATH03715707} that $H^2(X \setminus \{0\},\ZZ) \cong \Cl(X)$. Note, that the singularity of $X$ is rational, since $X$ is assumed to be log-terminal. Hence, the preconditions of \cite[Satz~6.1]{zbMATH03715707} are indeed fulfilled.
\end{proof}

Factorial cone singularities of complexity $1$ have been classified in dimension $2$ by Mori \cite{zbMATH03609829}, in dimension $3$ by Ishida \cite{zbMATH03627320} and in arbitrary dimension by Hausen-Herppich-S\"u{\ss} \cite{tfanos}. For the special case of an isolated cone singularities The results  may be summarised as follows.

\begin{proposition}[{\cite[Thm.~6.5.]{tsing}}]
\label{prop:isolated-factorial}
  The only isolated factorial cone singularities of complexity $1$ and dimension at least $3$ are the following.
  \begin{enumerate}
  \item The affine spaces $\CC^n$,
  \item the \emph{Brieskorn-Pham singularities} in $\CC^4$, given by the equation \[x_1x_2+x_3^p+x_4^q=0,\] with $p, q > 1$ being co-prime, \label{item:bp-sing}
  \item in $\CC^5$, given by the equation \[x_1x_2+x_3x_4 + x_5^p=0,\] 
    for $p > 1$,
  \item in $\CC^6$, given by the equation \[x_1x_2+x_3x_4 + x_5x_6=0.\]
  \end{enumerate}
\end{proposition}

\begin{corollary}
\label{cor:sasaki-sphere}
  Let $n$ be at least $3$ then every Sasaki metric structure with an isometric $(S^1)^{n-1}$-action on a sphere $S^{2n-1}$ arises as the link of one of the singularities in Proposition~\ref{prop:isolated-factorial}.
\end{corollary}

\begin{corollary}
   There are no irregular Sasaki-Einstein structures on $S^5$.
\end{corollary}
\begin{proof}
  By Corollary~\ref{cor:sasaki-sphere} all irregular Sasaki-Einstein structures on $S^5$ would arises as a link of a Brieskorn-Pham singularity as in Proposition~\ref{prop:isolated-factorial}~(\ref{item:bp-sing}). On the other hand, the Sasaki-Einstein structures on these manifolds are known to be quasi-regular, see \cite{collins2015sasaki}.
\end{proof}

\section{K-stability for Fano cone singularities}
\label{sec:k-stability-fano}
In this section we follow mostly the notation of \cite{zbMATH06868031} and \cite{collins2015sasaki}. Again we consider the weight decomposition of the affine coordinate ring under the action of the torus $T$.
\[\CC[X] = R = \bigoplus_{u \in \sigma^\vee \cap M} R_u.\]

For $\xi$ from the interior of $\sigma$ the index character $F(\xi,t)$ is defined by
\[F(\xi,t)= \sum_{u \in \sigma^\vee \cap M} e^{-t\langle u, \xi \rangle} \dim_\CC  R_u.\]
We obtain a meromorphic expansion for $F(\xi,t)$ as follows
\begin{equation}
    F(\xi,t)=\frac{a_0(\xi)(n-1)!}{t^n} + \frac{a_1(\xi)(n-2)!}{t^{n-1}} + O(t^{2-n}).\label{eq:index-char}
  \end{equation}

\begin{definition}
  The coefficent $a_0(\xi)$ is called the \emph{volume} of $(X,\xi)$ and will be denoted by $\vol(\xi)$.
\end{definition}
  
\begin{definition}
  The Futaki invariant for the pair $(X,\xi)$ and an element $v \in N_\RR$ is defined as follows
  \begin{equation}
    \label{eq:futaki-invariant}
    \Fut_\xi(X,v) = \frac{a_0(\xi)}{n-1}D_{-v}\left(\frac{a_1}{a_0}\right)(\xi)
+\frac{a_1(\xi)D_{-v}a_0(\xi)}{n(n-1)a_0(\xi)},
  \end{equation}
where $D_v$ denotes the directional derivative, i.e.
\[D_{-v} a_i(\xi)= \left.\frac{d}{ds}\right|_{s=0}\!\!\!\!a_i(\xi-sv).\]
\end{definition}

\begin{definition}
  A $T$-equivariant \emph{special degeneration} of a polarised Fano cone singularity $(X,\xi)$ is a pair $(\X,v)$ of a family $\X \to \CC$ with a $T$-action along the fibres and a $\CC^*$-action $v$ covering the standard $\CC^*$ on the base, such that
  \begin{enumerate}
  \item the generic fibres are isomorphic to $X$,
  \item the special fibre $\X_0$ is normal,
  \item both actions commute.
  \end{enumerate}
  If $\X \cong X \times \CC^*$, then the special degneration is called \emph{trivial}. 
\end{definition}
In the situation of a special degeneration as above there is a torus $T'$ acting on $Y$ generated by the commuting actions of $T$ and $\CC^*$. The embedding of tori induces an embedding of the corresponding groups of co-characters and their associated vector spaces. $N_\RR \hookrightarrow N_\RR'$. With this $(\X_0,\xi)$ becomes a polarised Fano cone singularity and the $\CC^*$-action $v$ can be seen as an element of $N'$.

\begin{definition}
  A polarised Fano cone singularity $(X, \xi)$ is called \emph{K-stable} if for every $T$-equivariant degeneration $(\X,v)$ with special fibre $Y$ one has $\Fut_\xi(\X_0,v) \geq 0$ with equality if and only if $\X \cong X \times \CC^*$. 
\end{definition}

The importance of the notion of K-stability for the study of Sasaki-Einstein structures comes from the following theorem by Collins and Sz\'ekelyhidi.

\begin{theorem}[{\cite[Thm. 1.1]{collins2015sasaki}}]
\label{thm:k-stab-Einstein}
  The K-stability of $(X,\xi)$ is equivalent to the existence of a Ricci flat K\"ahler cone metric on $(X,\hat \xi)$ for an appropriate rescaling $\hat \xi$ of $\xi$ or equivalently the to the existence of a Sasaki-Einstein structure with Reeb field $\hat \xi$ on the link of $X$.
\end{theorem}

\section{Fano cone singularities from polyhedral divisors}
\label{sec:polyhedral-divisors}
We shortly recall the description of torus actions on affine varieties from \cite{pre05013675}. We keep the notation from Secion~\ref{sec:spheres}, i.e. $T$ is a fixed algebraic torus of dimension $n$, $M$ and $N$ are its character lattice and the dual lattice, respectively, $M_\RR$ and $N_\RR$ denote the corresponding $\RR$-vector spaces.

In general a \emph{polyhedral divisor} with rational polyhedral \emph{tail cone} $\sigma \subset N_\RR$ on a semi-projective normal variety $Y$ is a formal sum
\[
\D = \sum_{Z} \D_Z \cdot Z.
\]
with $Z$ running over the prime divisors $Z \subset Y$ and $\D_Z$ being rational polyhedra in $N_\QQ$ with tail cone $\sigma$, i.e. $\sigma = \{v \in N_\RR \mid v+\D_Z \subset D_Z\}$ for all $Z$. For all but finitely many $Z$ the polyhedral coefficients $\D_Z$ coincide with the tail cone $\sigma$. We call these coefficients trivial. By the \emph{support} of $\D$ we mean the union of all $Z$, where $\D_Z$ is non-trivial. Note that in concrete examples we usually write down explicitly only the summands with non-trivial coefficients (as it is done for ordinary divisors).

For every polyhedral coefficient $\D_Z$ we denote its set of vertices by $\D_Z^{(0)}$. The \emph{multiplicity} of a vertex $v \in \D_Z^{(0)}$ is defined to be the minimal natural number $\mu(v) \in \NN$ such that $\mu(v)v$ becomes a lattice element.

A polyhedral divisor can be \emph{evaluated} at elements $u \in \sigma^\vee \cap M$ as follows
\[
\D(u) = \sum_{Z} \big(\min_{v \in \D_Z} \langle u , v \rangle\big)\cdot Z.
\]
The result is an ordinary $\QQ$-divisor on $Y$. A polyhedral divisor is called \emph{proper} if the following conditions are fulfilled 
\begin{enumerate}
\item $\D(u)$ is Cartier for every $u \in \sigma^\vee \cap M$,
\item $\D(u)$ is semi-ample for every $u \in \sigma^\vee \cap M$,
\item $\D(u)$ is big for every $u$ from the interior of $\sigma^\vee$.
\end{enumerate}

A proper polyhedral divisor defines an affine variety $X=\X(\D)$ of dimension $\dim Y + \dim T$ via
\[
\CC[X] = R = \bigoplus_{u \in \sigma^\vee \cap M} H^0(Y, \CO(\lfloor \D(u)\rfloor)).
\]
Moreover, the natural $M$-grading induces an effective $T$-action on $X$. By \cite{pre05013675} every normal affine variety with torus action can be obtained by this construction. From the construction it is clear that $R_0 = \CC$ if and only if $Y$ is actually projective and that the Reeb cone is given as the interior if $\sigma$. We see that the dimension of $Y$ equals the complexity of the torus action.

We now specialise to the case of Fano cone singularities of complexity $1$. Here, $Y$ has to be a curve. Moreover, by \cite[Cor. 5.8.]{tsing} log-terminality implies that $Y \cong \PP^1$. In this situation the conditions above simplify drastically. Given a polyhedral divisor $\D = \sum_{y \in \PP^1}\D_y \cdot y$ its properness corresponds simply to the condition 
\[\deg(\D) := \sum_{y} \D_y \subsetneq \tail \D.\]
Here, the summation over the polyhedral coefficients uses the Minkowski addition of polyhedra.

Let $\D$ be a proper polyhedral divisor on $\PP^1$ and assume that $\D$ is supported in $y_1, \ldots ,y_r \in \PP^1$. Then by \cite[Prop. 4.4]{tsing} the $\QQ$-Gorenstein condition is equivalent to the fact that the following system of linear equations has a (unique) solution $(a_{y_1}, \ldots a_{y_r}, \u) \in \QQ^r \times M_\QQ$.
\begin{align}
  \forall_{i=1,\ldots, r}\forall_{v \in \D_{y_i}^{(0)}} \colon \langle \u, v \rangle &= a_{y_i} - \frac{\mu(v)-1}{\mu(v)},\label{eq:canonical-weight-vertical}\\
  \forall_{\text{rays }\rho \subset (\tail \D \setminus \deg \D)} \colon \langle \u, v_\rho \rangle& = 1,\nonumber \\
  \sum_{i=1}^r a_{y_i} &=2 \nonumber
\end{align}
Here, $\rho$ runs over the rays of the tail cone, which do no intersect the degree polyhedron $\deg \D$ and $v_\rho$ denotes the primitive lattice generator of $\rho$.

By \cite[Cor. 5.8]{tsing} for $\X(\D)$ being log-terminal is equivalent to the fact that 
\begin{equation}
  \label{eq:log-terminal}
  \sum_{y\in \PP^1} \left(1-\frac{1}{\max \{\mu(v) \mid v \text{ is a vertex of } \D_y\}}\right) < 2.
\end{equation}
As a consequence of this condition $\u$ lies in the interior of the weight cone $(\tail \D)^\vee$. 

A polyhedral divisor $\D$ on $\PP^1$ fulfilling the above conditions defines a Fano cone singularity $\X(\D)$ and every Fano cone singularity of complexity $1$ arises this way. The Reeb cone is given as the interior of $\tail(\D)$.

\begin{definition}
  We call this uniquely determined element $\u \in M_\QQ$ from (\ref{eq:canonical-weight-vertical}) the \emph{canonical weight} of $\D$.  Similarly, for a toric $\QQ$-Gorenstein singularity given by a cone $\sigma \subset N_\RR$ the \emph{canonical weight} is defined to be the element $\u \in M_\RR$, such that $\langle \u, v_\rho \rangle = 1$ for all primitive generators $v_\rho$ of rays $\rho$ of $\sigma$. 

Infact, in both cases $\u$ is the weight of a generator of the canonical module $\omega_R$ of $R=\CC[X]$.
\end{definition}

Following \cite{is17} we are now going to describe equivariant special degenerations of $\X(\D)$ in terms of the polyhedral divisor $\D$.

\begin{definition}
  Given a polyhedral divisor $\D$ on $\PP^1$ then a choice $y \in \PP^1$ is called admissible if for every $u \in (\tail \D)^\vee \cap M$ one has $\min_{v \in \D_z} \langle u , v \rangle \notin \ZZ$ for at most one $z \neq y$.
\end{definition}

\begin{proposition}[{\cite[Thm. 4.3]{is17}}]
\label{prop:degenerations}
  The special fibre of a non-trivial equivariant special degenerations of $\X(\D)$ is an affine toric variety corresponding to a cone 
\[\sigma_y= \pos\left(\left(\tail \D\times \{0\}\right)\; \cup\; \left(\D_y \times \{1\}\right)\; \cup \; \left(\left({\textstyle \sum_{z \neq y} \D_y} \right) \times \{-1\}\right) \right).\]
Where $y \in \PP^1$ is admissible. The induced $\CC^*$-action is given by an element of $N \times \ZZ_{< 0}$.
\end{proposition}

\begin{remark}
  If $y \in \PP^1$ is an admissible choice with respect to $\D$ then the canonical weight $\u_y$ for $\sigma_y$ has the form $\u_y=(\u,a_y+1)$, where $\u \in M_\RR$ is the canonical weight for $\D$ and $a_{y_i}$ is as defined and $a_y := 0$ if $y \notin \{y_1, \ldots, y_r\}$.
\end{remark}

For a cone $\sigma \subset N_\RR$ consider the following polytope obtained by truncation of the dual cone
\[
\sigma^\vee(\xi)=\left\{u \in \sigma^\vee \mid \langle u, \xi \rangle \leq 1 \right\}.
\]

By \cite{zbMATH05159638} we can use this to calculate the volume of $\xi$.
\begin{proposition}
\label{prop:vol-toric}
  For a $\QQ$-Gorenstein toric Fano cone singularity. The volume functional $\vol: N_\RR \to \RR$ is given by
\[\vol(\xi) = \vol \sigma^\vee(\xi)\]
\end{proposition}

Similarly we have a statement in the situation of a complexity-$1$ torus action.
\begin{proposition}
\label{prop:vol-pdiv}
  For a $\QQ$-Gorenstein Fano cone singularity, given by a polyhedral divisor $\D$ on $\PP^1$. The volume functional $\vol: N_\RR \to \RR$ is given by
\[\vol(\xi) = \vol \sigma_y^\vee(\xi,0) := \vol \sigma_y^\vee((\xi,0)),\]
for any $y \in \PP^1$.
\end{proposition}
\begin{proof}
  On the one hand, this can be derived from the general fact that the asymptotic behaviour of the dimension of the homogeneous components does not change under equivariant flat deformations. On the other hand, we can see this as a special case of \cite[Lem.~3.12]{li2017stability} which calculates $\vol(\xi)$ as $\vol \hat\sigma^\vee(\xi)$ for a certain cone $\hat \sigma$ constructed from a $(N\times \ZZ)$-valued valuation. Indeed, by \cite[Cor.~5.5]{im17} we have $\sigma_y = \hat \sigma$ for a particular choice of a such a valuation.
\end{proof}

Hence, the volume can be effectively calculated on the toric degenerations, discussed in Section~\ref{sec:polyhedral-divisors}.

\begin{example}[$\QQ$-factorial toric case]
\label{exp:q-factorial-toric-1}
  Consider the case of a full-dimensional simplicial cone $\sigma \subset N_\RR$. Let $u_1, \ldots, u_n \in M$ be the primitive lattice generators of the dual cone. Then
\begin{equation}
\vol(\xi) = \vol \sigma^\vee(\xi) =  \left|\det \left(\frac{u_1}{\langle u_1, \xi  \rangle},\ldots,\frac{u_n}{\langle u_n, \xi  \rangle}\right)\right|
= \frac{|\det(u_1,\ldots,u_n)|}{\langle u_1, \xi  \rangle \cdots \langle u_n, \xi  \rangle}.\label{eq:volume-simplicial}
\end{equation}
Note, that just for convenience here and in the following we consider the normalised lattice volume. This differs from the Lebesgue volume only by the constant factor $n!$.
\end{example}

\medskip 

 Assume that $(Y,\xi)$ is a polarised Fano cone singularity with canonical weight $\u$ and $\langle \u, v \rangle = 0$. Then according to \cite[Sec. 6.2]{collins2015sasaki} and \cite[Sec.~2.5]{li2017stability}  we have
\begin{equation}
  \Fut_\xi(Y,v)= D_{-\hat v}\vol(\hat \xi) \label{eq:futaki-projection},
\end{equation}
where $\hat \xi = \nicefrac{\xi}{\langle \u, \xi \rangle}$ and $\hat v$ 
is some positive recaling of $v- \frac{\langle \u, v \rangle}{\langle \u, \xi \rangle} \xi \;\in \u^\perp$.

Since in the toric case non-trivial equivariant special degenerations do not exist K-stability is equivalent to the condition that $D_{-v}\vol(\xi)$  vanishes for all $v \in \u^\perp$. On the other hand, it was observed by \cite{zbMATH05294705} that the volume functional is convex and proper on $[\u = 1] \cap \sigma$. As such it has a unique critical value. As it was pointed out in \cite{collins2015sasaki} in this way it is possible to recover the following result from \cite{zbMATH05682662}.

\begin{proposition}
  The link of a $\QQ$-Gorenstein toric singularity admits a Sasaki-Einstein structure.
\end{proposition}

\begin{example}[$\QQ$-factorial toric case (continued)]
\label{exp:q-factorial-toric-2}
  Consider again the case of $X$ to be toric and $\QQ$-factorial, i.e. the corresponding cone $\sigma$ is spanned by an integral $N_\RR$-basis $v_1, \ldots, v_n$. Let us denote the elements of the dual basis by $u_1, \ldots, u_n$ of $M_\RR$. In the light of Example~\ref{exp:q-factorial-toric-1} we see that $D_{(v_j-v_i)}\vol(\xi)=0$ if and only if
\[0= \left.\frac{d}{ds}\right|_{s=0} \prod_k \langle u_k , \xi + s(v_i-v_j)\rangle.\]
Now, one calculates
\begin{align*}
0 = \left.\frac{d}{ds}\right|_{s=0} \prod_k \langle u_k , \xi + s(v_i-v_j)\rangle &= 
\prod_{k\neq i} \langle u_k, \xi \rangle - \prod_{k\neq j} \langle  u_k, \xi \rangle 
\end{align*}
Since we have $\langle  u_i, \xi \rangle > 0$ for all $i=1,\ldots, n$,  we may divide by
$\prod_{k\neq i,j} \langle  u_k, \xi \rangle$ and obtain $\langle  u_j, \xi \rangle - \langle u_i, \xi \rangle = 0$. Hence, $(X,\xi)$ is K-stable if and only if  $\langle  u_i, \xi \rangle = \langle u_j, \xi \rangle$, but this is equivalent to the fact that $\xi$ is a multiple of $\sum_{i=1}^n v_i$. It follows, that a toric Sasaki-Einstein structure on a $\QQ$-homology sphere is necessarily quasi-regular.
\end{example}

For the case of a Fano cone singularity $X=\X(\D)$ obtained from a polyhedral divisor $\D$ on $\PP^1$ we obtain the following theorem.

\begin{theorem}
\label{thm:kstab-pdiv}
  Given a polarised Fano cone singularity $(X,\xi)$ coming from a polyhedral divisor $\D$ with canonical weight $\u$. Then $(X,\xi)$ is K-stable if and only if 
\[
D_{(v,0)} \vol(\sigma_y^\vee(\hat \xi,0)) = 0
\]
for some $y \in \PP^1$ and all $v \in \u^\perp$ and
\[
D_{v'} \vol(\sigma_y^\vee(\hat \xi,0)) > 0
\]
for every admissible choice $y \in \PP^1$ and $v' \in \u_y^\perp \cap (N_\QQ \times \ZZ_{> 0}).$ 
\end{theorem}



\section{Non-existence of irregular Sasaki-Einstein metrics on rational homology $5$-spheres}
\label{sec:hom-spheres}
In this section we specialise to the case of three-dimensional isolated Fano cone singularities $X$, i.e. those admitting five-dimensional manifolds as their links. Let us assume that the link admits an irregular Sasaki-Einstein structure. Hence, the corresponding Reeb field  $\xi$  generates a torus $T$ of dimension at least $2$, which acts effectively on $X$. On the other hand, from the considerations in Example~\ref{exp:q-factorial-toric-2} we know, that $X$ cannot be toric. Hence, we are left with the case of a complexity-$1$ torus action.

Let $X=\X(\D)$ be a three-dimensional Fano cone singularity given by a polyhedral divisor $\D$. Recall that $\tail(\D) \subset N_\RR \cong \RR^2$ coincides with the closure of the Reeb cone and is necessarily full-dimensional. Then $\tail(\D)$ is bounded by exactly two rays $\rho^1$ and $\rho^2$ with primitive lattice generators $v_1$ and $v_2$. Consequently every polyhedral coefficient $\D_y$ has two facets of the form $v_y^1 + \rho^1$ and $v_y^2 + \rho^2$, where $v_y^1$ and $v_y^2$ may coincide. The other facets of $\D_y$ have the form of a line segment $\overline{vw}$.

In this case we have the following criterion for isolated singularities.
\begin{proposition}
\label{prop:isolated}
  Let $\D$ be a polyhedral divisor on $\PP^1$ with two-dimensional tail cone $\sigma=\pos(v_1,v_2) \subset M_\RR \cong \RR^2$. Then the corresponding three-dimensional affine $T$-variety $\X(\D)$ has an isolated singularity if and only if the following conditions hold

  \begin{enumerate}
  \item For every facet $\overline{vw}$ of polyhedral coefficient $\D_y$ the vectors $\mu(v)(v,1)$ and $\mu(w)(w,1)$ can be complemented to a lattice basis of $N \times \ZZ$,\label{item:isolated-compact-facet}
  \item for $i \in \{1,2\}$ and $\rho_i \cap \deg \D = \emptyset$ and $v_y^i + \rho^i$ being a facet of a polyhedral coefficient $\D_y$ the vectors $(v^i,0)$ and $\mu(v_y^i)(v_y^i,1)$ can be complemented to a lattice basis of $N \times \ZZ$,\label{item:isolated-extremal}
  \item for $i \in \{1,2\}$ and $\rho_i \cap \deg \D \neq \emptyset$ then with the exception of (at most) two vertices $v_z^i, v_{z'}^i$ the vertices $v_y^i$ are integral and the vectors $\mu(v_z^i)(v_z^i,1)$ and $\mu(v_{z'}^i)(\sum_{y \neq z} v_y^i, 1)$ can be complemented to a lattice basis of $N \times \ZZ$.\label{item:isolated-non-extremal}
  \end{enumerate}
\end{proposition}
\begin{proof}
  By Theorem~5.4 in \cite{tsing} we have to consider the polyhedral divisors
\[\D^u = \sum \face(\D_y,u) \otimes y\]
for elements $u \in (\tail \D)^\vee \setminus \{0\}$, such that at least
one of the coeffcients $\face(\D_y,u)$ has codimension one. Here, $\face(\D_y,u)$ denotes the face of the polyhedron $\D_y$ where the linear form $\langle u, \cdot \rangle$ is minimised. Assume that $\overline{vw} = \face(\D_y,u)$. Then $\tail \D^u = \tail(\overline{vw}) = 0$ and, in particular, $\deg \D^u \not \subset \tail \D^u$.
Hence, we are in in situation (ii) of \cite[Thm.~5.4]{tsing}. In which case, by Theorem \cite[Thm.~5.3]{tsing} $\delta_y=\RR_{\geq 0} \cdot (\overline{vw} \times \{1\})$ need to be a \emph{regular cone}, i.e. the primitive lattice generators of it's rays can be complemented to a lattice basis. But the rays of $\delta_y$ are precisely $\mu(v)(v,1)$ and $\mu(w)(w,1)$. Hence,
we obtain condition (\ref{item:isolated-compact-facet}).

Now, assume $v_y^i + \rho^i = \face(\D_y,u)$. Assume first $\rho^i \cap \deg \D = \emptyset$ then $\deg \D^u = \face(\deg \D,u) \not \subset \rho^i = \face(\tail \D,u)$. Hence, we are again in the situation (ii) of \cite[Thm.~5.4]{tsing}. In which case, by \cite[Thm.~5.3]{tsing} the closure of $\RR_{\geq 0} \cdot ((v^i_y + \rho^i) \times \{1\})$, which is spanned by $(v^i_y,1)$ and $(v^i,0)$ has to be a regular cone, i.e. $\mu(v^i_y)(v^i_y,1)$ and $(v^i,0)$ can be complemented to a lattice basis. Hence, condition~(\ref{item:isolated-extremal}) follows.

It remains consider the situation when $\rho^i \cap \deg \D \neq \emptyset$. This implies that $\rho^i \cap \deg \D = \face(\deg \D, u) = \deg \D^u$. In particular, $\deg \D^u \subset \rho^i$. Hence, we are in the situation (i) of \cite[Thm.~5.4]{tsing}. Similar to the other cases \cite[Prop.~5.1]{tsing} implies condition (\ref{item:isolated-non-extremal}) above.

The inverse implication follows analogously.
\end{proof}

\begin{proposition}
\label{prop:q-factorial-isolated}
  A $\QQ$-factorial three-dimensional cone singularity with maximal torus $T=(\CC^*)^2$ and isolated singularity is given by a polyhedral divisor supported on three points on $\PP^1$, which we denote by $0$, $\infty$ and $1$, fulfilling the following properties
  \begin{enumerate}
  \item $\D_{\infty} = \overline{v_\infty  v'_\infty} + \tail \D$, with 
$v_\infty,v_\infty' \in N$ being a lattice elements and $(v_\infty-v'_\infty)$ being a primitive lattice element. 
  \item $\D_{0}=v_0 + \tail \D$ and $\D_1 = v_1 + \tail \D$ for some $v_0,v_1 \in N_\QQ$.
  \end{enumerate}
\end{proposition}
\begin{proof}
  By \cite[Cor. 3.15]{tdiv} having $\QQ$-factoriality in this situation implies that we are in one of the following two situations
  \begin{enumerate}
  \item there is exactly one polyhedral coefficient of the form
    $\overline{v v'} + \tail \D$ and all others are just translated tail cones. Moreover, $\deg \D$ intersects bot rays of $\tail \D$.\label{item:no-extremal-ray}
  \item All coefficients are translated tail cones and $\deg \D$ intersects exactly one ray $\rho$ of $\tail \D$.\label{item:extremal-ray}
  \end{enumerate}
  Let's assume we are in situation (\ref{item:extremal-ray}). The fact that the $2$-torus action $\X(\D)$ is assumed to be maximal implies by \cite[Sec.~11]{pre05013675} that there are at least three polyhedral coefficients, which are not just lattice translations of $\deg \D$. However, this would violate condition (~\ref{item:isolated-non-extremal}) of Proposition~\ref{prop:isolated} for the ray $\rho$.

  Hence, we are in necessarily in situation (\ref{item:no-extremal-ray}). Again he fact that  the $2$-torus action $\X(\D)$ cannot be extended to a $3$-torus action implies that there are at least three polyhedral coefficients, which not just integral translations of $\deg \D$. Beside the one coefficient of the form $\overline{v v'} + \tail \D$ there are at least two further coefficients which are non-integral translates of the tail cone. However,  (\ref{item:isolated-non-extremal}) of Proposition~\ref{prop:isolated} implies that there cannot be more than two of such elements and
that $v$ and $v'$ are lattice elements. Moreover, by condition~(\ref{item:isolated-compact-facet}) of
 Proposition~\ref{prop:isolated} $v-v'$ has to be a primitive lattice element.
\end{proof}

\begin{theorem}
 A Sasaki-Einstein structure on the link of a $\QQ$-factorial three-dimensional Fano cone singularity of complexity $1$ is necessarily quasi-regular.
\end{theorem}
\begin{proof}
Consider a polyhedral divisor fulfilling the conditions of Proposition~\ref{prop:q-factorial-isolated}. We set $v = v_\infty -v_\infty'$. Then the cone
\[\sigma_1 = \pos\{(v_0,-1),(v_0+v,-1),(v_1,1)\} \subset N_\RR \times \RR\]
describes the central fibre of a special degeneration as discussed in Section~\ref{sec:polyhedral-divisors}. Assume that $\u$ is the canonical weight of $\D$. By (\ref{eq:canonical-weight-vertical}) this implies that $\langle \u, v_\infty \rangle = \langle \u, v'_\infty \rangle$. 
Since $\dim N_\RR = 2$, the orthogonal complement $\u^\perp$ is generated by $v=v_\infty-v'_\infty$. On the other hand, the dual cone $\sigma^\vee_0$ is generated by three primitive lattice elements $w_1$, $w_2$ and $w_3$. One of them, say $w_1$, will be orthogonal to $(v_0+v,-1)$ and $(v_0,-1)$ and, hence, to their difference $(v,0)$, as well. This implies 
\[\langle w_1, (v,0) \rangle = \langle p_1(w_1), v\rangle = 0.\]
Here, $p_1\colon M_\RR \times \RR \to M_\RR$ denotes the projection. Then for dimension reasons 
\begin{equation}
  p_1(w_1) = \lambda \u\label{eq:multiple-of-canonical-weight}
\end{equation}

for some $\lambda \in \RR$.

Assume that $\xi \in N_\RR$ is the Reeb field of a Sasaki-Einstein structure on the link of $X$ and $v \in \u^\perp$. Then by Theorem~\ref{thm:k-stab-Einstein} and Theorem~\ref{thm:kstab-pdiv} we must have 
\[D_{-v}\vol(\hat \xi) = D_{(-v,0)}\vol (\sigma_0^\vee (\hat \xi,0)=0.\] On the other hand, $\sigma_0$ is a simplicial cone. Hence, as seen in Example~\ref{exp:q-factorial-toric-1} for $\vol(\hat \xi+sv)$ we obtain
\[\vol (\sigma^\vee(\hat \xi-sv,0)) = \frac{6|\det(w_1,w_2,w_3)|}{\langle w_1,(\hat \xi+sv,0) \rangle \langle w_2,(\hat \xi-sv,0) \rangle\langle w_3,(\hat \xi-sv,0) \rangle}.\]
Hence, $D_v\vol(\hat \xi)$ vanished if and only if 
\begin{equation}
\left. \frac{d}{ds}\right|_{s=0} \langle w_1,(\hat \xi-sv,0) \rangle \langle w_2,(\hat \xi-sv,0) \rangle\langle w_3,(\hat \xi-sv,0) \rangle=0.\label{eq:vol-deriv}
\end{equation}
On the other hand  $\langle w_1,(\hat \xi-sv,0) \rangle = \langle p_1(w_1),\hat \xi -sv \rangle = \lambda$ by the definition of $\hat \xi$ and (\ref{eq:multiple-of-canonical-weight}). This implies, that the left-hand-side of (\ref{eq:vol-deriv}) is a linear in $\hat \xi$ with coefficients in $\QQ$. Hence, the Reeb field $\hat \xi$, as the unique solution of (\ref{eq:vol-deriv}), must be rational and the corresponding Sasaki-Einstein structure on the link of must be quasi-regular. 
\end{proof}

\section{A family of non-toric irregular Sasaki-Einstein metrics}
\label{sec:main-example}
In this section we prove the following theorem.
\begin{theorem}
\label{thm:family-example}
  For every $k \in \NN \setminus \{3\}$ there is a family of irregular Sasaki-Einstein structures on $k \cdot (S^2 \times S^3)$. For $k > 1$ these structures admit non-trivial moduli.
\end{theorem}

For a fixed odd $k \in \NN$ we consider the p-divisors $\D$ on $\PP^1$ with tail cone $\sigma = \rho^1 + \rho^2$, where $\rho^1 = \RR_{\geq 0}\cdot (-1,1)$ and $\rho^1 = \RR_{\geq 0}\cdot (15k-4,8)$. Its coefficients are given as follows.
\begin{align*}
 \D_0 &= \frac{1}{5}(2,1) + \sigma,\\
 \D_\infty &= \frac{1}{3}(-2,1) + \sigma,\\
 \D_{y_1}= \ldots = \D_{y_k} &= \overline{(0,0)(1,0)} + \sigma.\\
\end{align*}
See Figure~\ref{fig:pdiv} for the case $k=1$.
\begin{figure}[h]
  \tikzset{x=0.7cm,y=0.7cm}
  \centering
  \begin{subfigure}{0.2\linewidth}
    \begin{tikzpicture}
      \begin{scope}
        \clip(-1,-0.2) rectangle (2,1); 
        \draw[thin,step=0.2,help lines,color=lightgray] (-2,-0.2) grid (2,1); \draw[fill] (0,0) circle (0.05);
        \draw[fill=lightgray]
        (0.4-1,1+0.2)--(0.4,0.2)--(11+0.4,8+0.2);
      \end{scope}
    \end{tikzpicture}
    \caption{$\D_0$}
  \end{subfigure}
  \begin{subfigure}{0.2\linewidth}
    \begin{tikzpicture}
      \begin{scope}
        \clip(-2,-0.2) rectangle (1,1);
        \draw[thin,step=0.2,help lines,color=lightgray] (-2,-0.2) grid (2,1);
        \draw[fill=lightgray] (-1-2/3,1+1/3)--(-2/3,1/3)--(11-2/3,8+1/3);
        \draw[fill] (0,0) circle (0.05);
      \end{scope}
    \end{tikzpicture}
    \caption{$\D_\infty$}
  \end{subfigure}
  \begin{subfigure}{0.2\linewidth}
    \begin{tikzpicture}
      \begin{scope}
        \clip(-1,-0.2) rectangle (2.5,1);
        \draw[step=0.2,help lines,color=lightgray] (-2,-0.2) grid (2.5,1);
        \draw[fill=lightgray] (-1,1)--(0,0)--(1,0)--(11+1,8);
        \draw[fill] (0,0) circle (0.05);
      \end{scope}
    \end{tikzpicture}
    \caption{$\D_1$}
  \end{subfigure}
  \begin{subfigure}{0.2\linewidth}
    \begin{tikzpicture}
      \begin{scope}
        \clip(-1.2,-0.2) rectangle (2.5,1);
        \draw[thin,step=0.2,help lines,color=lightgray] (-2,-0.2) grid (2,1);
        \draw[fill=white] (-2,2)--(0,0)--(11,8);
        \draw[fill=lightgray] (-1-2/3+2/5,1+1/3+1/5)--(-2/3+2/5,1/3+1/5)--(1-2/3+2/5,1/3+1/5)--(1-2/3+2/5+11,1/3+1/5+8);
        \draw[fill] (0,0) circle (0.05);
      \end{scope}
    \end{tikzpicture}
    \caption{$\deg \D \subsetneq \sigma$}
  \end{subfigure}
  \caption{Polyhedral coefficients and degree of $\D$ for $k=1$}
  \label{fig:pdiv}
\end{figure}
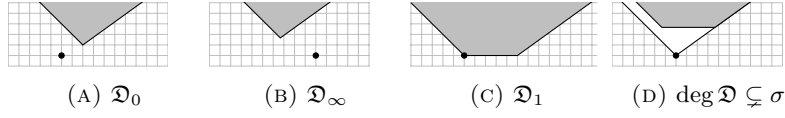

First note that $\D$ is proper for every $k$. Indeed, we have 
 \begin{equation}
\label{eq:degree}
\deg \D= \overline{\left(-\frac{4}{15},\frac{8}{15}\right)
    \left(\frac{15k- 4}{15},\frac{8}{15}\right)} + \sigma.
\end{equation}
This is a proper subset of $\sigma$. Hence, for a choice of $k$ distinct points $y_1, \ldots, y_k\in \CC^*$ the above data defines an affine variety $X_k =\X(\D)$. Varying the point configuration $y_1, \ldots, y_k$ leads to a $(k-1)$-dimensional family, which are pairwise equivariantly non-isomorphic, see \cite[Cor.~8.12]{pre05013675}. In the following we will see that for odd $k$ the $X_k$ are indeed isolated Fano cone singularities.

\begin{lemma}
  For every odd $k \in \NN$ the singularities $X_k$ are isolated.
\end{lemma}

\begin{proof}
  We are using Proposition~\ref{prop:isolated}. Recall the degree of $\D$ from (\ref{eq:degree}).
  One sees that $\deg \D \cap \rho^1 \neq \emptyset$ and
  $\deg \D \cap \rho^2 = \emptyset$. Hence, we have to check
  condition~(\ref{item:isolated-extremal}) for $\rho^1$ and
  condition~(\ref{item:isolated-non-extremal}) for
  $\rho^2$. Condition~(\ref{item:isolated-compact-facet}) has to be
  checked for the coefficients $\D_{y_i}$.  \smallskip

\paragraph*{Condition~(\ref{item:isolated-compact-facet}):}
For every of the coefficients $\D_{y_i}$ we have to consider the
lattice elements $(0,0,1)$ and $(1,0,1)$. Those can indeed be
complemented to a lattice basis, e.g. by $(0,1,0)$.  \smallskip

\paragraph*{Condition~(\ref{item:isolated-extremal}):}
We have to consider the pairs of vectors $(-1,1,0), (2,1,5)$
for $\D_0$ and $(-1,1,0), (-2,1,3)$ for $\D_\infty$ and
$(-1,1,0), (0,0,1)$ for $\D_{y_1}= \ldots = \D_{y_k}$. Each of these
pairs can be complemented to a lattice basis by $(0,-1,-2)$,
$(0,1,-2)$ and $(0,1,0)$, respectively.  \smallskip

\paragraph*{Condition~(\ref{item:isolated-non-extremal}):}
We have $v^2_0=\frac{1}{5}(2,1)$, $v^2_\infty=\frac{1}{3}(-2,1)$,
$v^2_{y_1}=\ldots=v^2_{y_k}=(1,0)$. Hence, only two vertices are
non-integral and the two vectors $5\cdot(v^2_0,1)=(2,1,5)$ and
$3\cdot(v_\infty + \sum_i v^2_i,1)=(3k-2,1,-3)$.  The maximal minors
of $ \left(\begin{smallmatrix}
    3k-2 & 1& -3\\
    2 & 1& 5
  \end{smallmatrix}\right)
$
are $3k-4$, $15k-4$ and $8$. For odd $k$ the greatest common divisor
of the minors is $1$ and the two vectors can be complemented to a
lattice basis.

Now Proposition~\ref{prop:isolated} implies that the cone singularity
is indeed isolated.
\end{proof}

\begin{lemma}
  The singularity $X_k$ are Gorenstein and log-terminal.
\end{lemma}

\begin{proof}
  The choices $\u=(0,1)$ and $a_0=a_\infty=1$ and $a_{y_i}=0$ for
  $i=1,\ldots,k$. solves the equations
  (\ref{eq:canonical-weight-vertical}). It follows that $X=\X(\D)$ is
  $\QQ$-Gorenstein. Moreover, $X$ is also log-terminal, since
  $\nicefrac{4}{5} + \nicefrac{2}{3} + 0 < 2$.
\end{proof}

Next, we calculate a unique (up to scaling) canditate for a Reeb field of a Sasaki-Einstein structure.
Recall that $\hat \xi = \xi/\langle\u, \xi\rangle \in [\u = 1] \cap \sigma$. By the above this implies $\hat \xi = (x,1)$ with $-1 \leq x \leq \frac{k\cdot 15-4}{8}$. By Theorem~\ref{thm:kstab-pdiv} we have to find the value for $x$ such that $D_{(1,0,0)}\vol \sigma_y^\vee(\hat \xi,0) = 0$ for some $y\in \PP^1$.

Consider the cone $\sigma_0$ as in Proposition~\ref{prop:degenerations}, which  corresponds to one of the toric degenerations of $X$. We obtain
\begin{align*}
\sigma_0&=\pos((2, 1, 5),( -2+3k, 1, -3),(-2, 1, -3),(-1,1,0)),\\
\sigma_0^\vee&=  \pos((-8, 15k - 4, -3k + 4),(0, 3, 1),(3, 3, -1),(5, 5, -3)).
\end{align*}

In order to calculate the volume of $\sigma^\vee_0(\xi,0)=\sigma^\vee_0(x,1,0)$ we subdivide the cone $\sigma_0^\vee$ in two simplicial cones
\begin{align*}
\omega_1 &= \pos((-8, 15k - 4, -3k + 4),(0, 3, 1),(5, 5, -3)),\\
\omega_2 &= \pos((0, 3, 1),(3, 3, -1),(5, 5, -3)).
\end{align*}

Then $\vol \sigma_0^\vee(v) = \vol \omega_1(v) + \vol \omega_2(v)$ and we may apply (\ref{eq:volume-simplicial}) from Example~\ref{exp:q-factorial-toric-1} to each summand and obtain
\begin{align}
\label{eq:example-volume}
\vol \sigma^\vee(x,1,0) &= \frac{8 \, {\left(15 \, k + 4\right)}}{15 \, {\left(15 \, k - 8 \, x - 4\right)} {\left(x + 1\right)}} + \frac{4}{15 \, {\left(x + 1\right)}^{2}}\\
&= \frac{4}{15}\cdot\frac{30kx+45k+4}{(x+1)^2(15k-8x-4)} \nonumber
\end{align}

Differentiating by $x$ gives
\begin{align*}
\frac{4}{15}\cdot \frac{480kx^2-6(75k^2-200k-16)x+(-900k^2+480k+64)}{(x+1)^3(15k-8x-4)^2}.
\end{align*}
The unique root $x_0 \in [-1,\frac{15k-4}{8}]$ of the numerator is given by
\begin{equation}
  \label{eq:xi-solution}
  x_0 = \frac{225 \, k^{2} + \sqrt{225 \, k^{2} + 600 \, k + 144} {\left(15 \, k + 4\right)} - 600 \, k - 48}{480 \, k}.
\end{equation}
This is an irrational number for $k \neq 3$. Indeed, note that we have
\begin{equation}
\sqrt{225 k^{2} + 600  k + 144}=\sqrt{(15k + 20)^2 -256}.
\label{eq:square-completion}
\end{equation}
This equals $63$ for $k=3$. For $k \geq 9$ one has 
\begin{equation}
(15k+19)^2 < (15k+20)^2-256 < (15k+20)^2.\label{eq:x0-bounds}
\end{equation}
Hence, the square root is necessarily irrational for these values of $k$. For the remaining values irrationality can be checked case by case.

 Now, consider any positive recaling $\xi$ of $\hat \xi = (x_0,1)$. 
\begin{proposition}
  For every $k \in \NN$ the Fano cone singularities $(X_k, \xi)$ are K-stable.
\end{proposition}
\begin{proof}
  There are only two admissible choices for $y \in \PP^1$, namely $y=0$ and $y=\infty$.
  We first apply the criterion of Theorem~\ref{thm:kstab-pdiv} for the
  degeneration corresponding to $\sigma_0$. Similar to
  (\ref{eq:example-volume}) we calculate
  \begin{align}
    D_{(\hat \xi,1)} \vol \sigma^\vee(\hat \xi,1) &= \left.\frac{d}{dt}\right|_{t=0} \vol \left(\sigma^\vee_0(x_0,1,t)\right)\nonumber\\ 
& = \frac{24 {\left(15 k + 4\right)(3k-4)}-8 {\left(15 k  + 4\right)(15k-8x_0-4)}}{45{\left(15 k - 8 \, x_0 - 4\right)^2} {\left(x_0 + 1\right)}}+ \label{eq:line-1}\\
                                                  &+ \frac{72 \, {\left(15 \, k + 4\right)}-20(15k-8x_0-4)}{225 {\left(15  k - 8  x_0 - 4\right)} {\left(x_0 + 1\right)}^{2}} +\label{eq:line-2} \\
                                                  &+ \frac{56}{225 \, {\left(x_0 + 1\right)}^{3}}\nonumber
  \end{align}
  We have to show that this is a positive number for every
  $k \in \NN$. We do this explicitely for $k \geq 9$. The remaining
  cases can be checked in a similar fashion by bounding $x_0$ from above and below for a particular value of $k$. Plugging the lower bound of
  (\ref{eq:x0-bounds}) into (\ref{eq:xi-solution}) and using $k \geq 9$ gives
  the following estimate
   \[x_0 \geq \frac{450k^2 - 255k+28}{480k} \geq \frac{7}{8}k.\]
   Now, using this lower bound one obtains 
   $120k^2-928k-256$ and $920k+368$ as lower bounds for the numerators of (\ref{eq:line-1}) and (\ref{eq:line-2}), respectively. Those are both positive for $k \geq 9$. On the other hand, the denominators above are always positive due to the the condition $-1 \leq x_0 \leq 15k-4$.

  Similary we can check the positivity of the Futaki invariant for the degeneration corresponding to  $\sigma_\infty$. The corresponding cone and its dual are given as
  follows.
  \begin{align*}
    \sigma_\infty&=\pos((-2, 3, 1),(2,  -5, 1),(2+5k,-5,1),(-1,0,1)),\\
    \sigma_\infty^\vee&=  \pos((-8, -5k - 4, 15k - 4),(0, 1, 5),(5, 3, 5),(3, 1, 3)).
  \end{align*}
\end{proof}

\begin{lemma}
\label{lem:simply-connected}
  The punctured cones $X_k \setminus \{0\}$ are simply-connected.
\end{lemma}
\begin{proof}
  We use  \cite[Thm.~ 3.4]{2017arXiv171202172L} to calculate the fundamental group of $X_k \setminus \{0\}$. A set of generators is given by
\[
b_0,b_\infty,b_{1},\ldots, b_{k} \qquad t_1,t_2.
\]
Here, every point in the support provides one generator $b_i$ and the remaining two generators $t_1, t_2$ correspond to a lattice basis of $N$. Now, according to \cite[Thm.~ 3.4]{2017arXiv171202172L} we obtain the following relations
\begin{align}
  b_0 b_\infty \prod_i b_i\label{eq:fg-1}\\
  [b_i,t_j],\; [t_i,t_j] \label{eq:fg-2} \\
  t_1^{-1}t_2, \quad t_1^{15k-4}t_2^8 && \text{(from the rays of $\sigma$)} \label{eq:fg-3}\\
  t_1^2t_2b_0^5,\quad t_1^{-2}t_2b_\infty^3&& \text{(from the vertices of  $\D_0$, $\D_\infty$)} \label{eq:fg-4}\\
  b_i,\quad t_1b_i,\quad \text{ for } i=1,\ldots,k.&& \text{(from the vertices of  $\D_i$, $i=1,\ldots, k$)}\label{eq:fg-5}
\end{align}
Relation (\ref{eq:fg-5}) implies $t_1 = 1$ and $b_i=1$ for $i=1,\ldots,k$. Now, by (\ref{eq:fg-3}) we obtain $t_2=1$. Relation (\ref{eq:fg-4}) gives $b_0^5=1$ and $b_\infty^3=1$. Since by (\ref{eq:fg-1})  the equality $b_0=b_\infty^{-1}$ holds, we obtain
\[1=1^2 1^5=(b_\infty^3)^2b_0^5=b_\infty^6b_\infty^{-5}=b_\infty = b_0^{-1}.\]
Hence, the group is trivial.
\end{proof}

\begin{proof}[Proof of Theorem~\ref{thm:family-example}] It follows from the previous results in this section and Theorem~\ref{thm:k-stab-Einstein} that the links $L_k$ of our isolated Fano cone singularities $X_k$ admit Sasaki-Einstein structures. It remains to show, that the links $L_k$ are diffeomorphic to  $k(S^2 \times S^3)$. 
By Lemma~\ref{lem:simply-connected} the link $L_k$ is simply-connected as it is a deformation retract of  $X_k \setminus \{0\}$. Now, \cite[Cor.~3.15]{tdiv} gives $\Cl(X_k) \cong \ZZ^{k}$. As before we can apply \cite{zbMATH03715707} to conclude $H^2(L_k) \cong \ZZ^{k}$, as well.

On the other hand, among the simply connected $5$-manifolds from the Smale-Barden classification only $S^5$ and $k (S^2 \times S^3)$ can admit Sasaki-Einstein structures with $(S^1)^2$-symmetry, see \cite[Prop. 10.2.27]{zbMATH05243165}. Hence, the result follows from the fact that  $H^2(k(S^2 \times S^3)) \cong \ZZ^k$.
\end{proof}

\begin{remark}
\label{rem:cox-construction}
  There is an alternative way to describe the affine varieties $X_k$. By \cite[Thm.~4.8]{tcox} their \emph{Cox ring} is given as 
\[\mathcal{R}=\CC[S,T_0,T_\infty,T_1,\bar T_1,\ldots,T_k,\bar T_k]/I\]
 with $I$ being a complete intersection ideal generated by elements
\[
T_0^5+y_i \cdot T_\infty^3+ T_i\bar T_i \qquad \text{ for } i=1, \ldots, k. 
\]
\
This ring is naturally graded by $H \cong \ZZ^{k}$. To describe this grading we choose a basis $e_1,\ldots, e_k$ and set
  \begin{alignat}{2}
    \deg_H(T_1)&=-4e_1 + \sum_{i=2}^k e_k \hspace{2em} & \deg_H(\bar T_1)&=19e_1 - \sum_{i=2}^k e_k \label{eq:weights1}\\
    \deg_H(T_k)&=-e_k & \deg_H(\bar T_k)&=15e_1+e_k\\
    \deg_H(T_0)&=3e_1 & \deg_H(T_\infty)&=5e_1\\
    \deg_H(S)&=-8e_1 &  &\label{eq:weight4} 
  \end{alignat}
Then the affine coordinate ring $\CC[X_k]$ is given as the degree-$0$-component $\mathcal{R}_0$ with respect to this $H$-grading. Moreover there is an additional $M$-grading with
$M=\ZZ^2$ and $\deg_M(S)=(0,1)$ and $\deg(T_i)=(1,1)$ and $\deg(\bar T_i)=(-1,-1)$. This induces an $M$-grading on $\CC[X_k]=\mathcal{R}_0$ and, hence, a $2$-torus action on $X_k$. These considerations allow to study the quasi-affine varieties $X_k$ or $X_k \setminus \{0\}$, respectively, via the methods from \cite{zbMATH06717888}. The main idea is to obtain $X_k \setminus \{0\}$ by a quotient construction from the \emph{total coordinate space} $\overline{X}_k = V(I) \subset \CC^{2k+3}$. It should be possible to translate this into a Sasakian reduction (in the sense of \cite[Sec.~8.5]{zbMATH05243165}) of the smooth locus of the link $\overline{L}_k = \overline{X}_k \cap S^{4k+5}$ by the subtorus $T^{k} \subset T^{2k+3}$ given by the weights in (\ref{eq:weights1})-(\ref{eq:weight4}).
The other way around using the same subtorus one may first apply Sasakian reduction to the ambient $(4k+5)$-sphere and obtain $L_k$ equivariantly embedded into the corresponding toric object. For $k>1$ this toric objects will be singular, but $L_k$ does not intersect its singular locus. However, for $k=1$ this gives indeed an embedding of $L_1$ as a hypersurface of a compact toric Sasakian manifold.
\end{remark}



\bibliography{sasaki}
\bibliographystyle{halpha}
\end{document}